\documentclass[12pt,reqno]{amsart}

\usepackage{amsmath,amssymb,amscd,graphicx, amsthm,
  enumerate, hyperref} 
\calclayout

\usepackage{mathrsfs}
\usepackage{mathtools}
\usepackage[all]{xy}
\usepackage{amsmath}
\usepackage{color,xcolor}
\definecolor{darkred}{rgb}{1,0,0} 
\definecolor{darkgreen}{rgb}{0,0.8,0}
\definecolor{darkblue}{rgb}{0,0,1}

\hypersetup{colorlinks,
linkcolor=darkblue,
filecolor=darkgreen,
urlcolor=darkred,
citecolor=darkgreen}

\numberwithin{equation}{section}

\theoremstyle{plain}

\theoremstyle{plain}
\newtheorem{theorem}{Theorem}
\numberwithin{theorem}{section}
\newtheorem{proposition}[theorem]{Proposition}
\newtheorem{lemma}[theorem]{Lemma}

\theoremstyle{definition}

\theoremstyle{definition}
\newtheorem{remark}[theorem]{Remark}

\newcommand{\Real}{\mathrm{Re}}

\newcommand{\interior}[1]{%
  {\kern0pt#1}^{\mathrm{o}}%
}

\newtheoremstyle{named}{}{}{\itshape}{}{\bfseries}{.}{.5em}{\thmnote{#3 }#1}
\theoremstyle{named}

\DeclareMathOperator{\rad}{rad}

\DeclareMathOperator{\Dom}{Dom}

\newcommand{\norm}[1]{\left\lVert#1\right\rVert}

\newcommand{\R}{\mathbb{R}}

\newcommand{\C}{\mathbb{C}}

\title[Algorithms for subelliptic multipliers in $\mathbb{C}^2$]{Algorithms for subelliptic multipliers in $\mathbb{C}^2$}
\author{Martino Fassina}
\address{Department of Mathematics, University of Illinois at Urbana-Champaign, 1409 West Green
Street, Urbana, IL 61801, USA}
\email{fassina2@illinois.edu}

\begin{document}
\begin{abstract} We give examples of pseudoconvex domains of finite type in $\C^2$ where the Kohn algorithm for subelliptic estimates fails to yield an effective lower bound for the order of subellipticity in terms of the type. We show how to modify the algorithm to obtain an effective procedure to prove subellipticity on domains of finite type in $\C^2$ with real analytic boundary satisfying a condition slightly stronger than pseudoconvexity. We close with a generalization to higher dimensions.
\end{abstract}
\maketitle
\section{Introduction}
The theory of subelliptic estimates for the $\bar\partial$-Neumann problem in $\C^2$ is completely understood. For a pseudoconvex domain $\Omega$ in $\C^2$ with smooth boundary $b\Omega$, a subelliptic estimate holds at a boundary point $p$ if and only if $p$ is a point of finite type. Moreover, letting $\tau$ be the value of the type of $b\Omega$ at $p$, the order of subellipticity is precisely $1/\tau$. This statement combines deep work carried out in the 1970s by several authors, including Kohn \cite{K72}, Greiner \cite{G74}, Rothschild and Stein \cite{RS76}.

After these results in dimension two were established, Kohn introduced a procedure to prove subelliptic estimates in general dimensions \cite{K79}. His algorithm, based on the notion of subelliptic multipliers, remains the object of active research. In particular, it has recently been observed that the algorithm does not yield a lower bound for the order of subellipticity in terms of the type \cite{CD10}. This phenomenon of ``lack of effectiveness" has been intensively studied, and variants of the algorithm have been proposed to make it effective in some classes of domains \cite{CD10,S10,S17,KZ18}.
 
Since the results on subellipticity in $\C^2$ are conclusive and independent of Kohn's theory of multipliers, the behavior of the Kohn algorithm in $\C^2$ appears not to have been investigated. In this paper we first show that ineffectiveness occurs even in the two-dimensional situation. We then provide a modified algorithm that is effective on a fairly general class of pseudoconvex domains in $\C^2$ with real analytic boundary.  

Effectiveness fails when, in forming an ideal of subelliptic multipliers $I_j$ as prescribed by the algorithm, a radical is taken whose order cannot be estimated by a function of the type \cite{H08,CD10,S10,KZ18} (See Section \ref{s2} for the definition of the order of a radical). In order to study this phenomenon, we will keep track of the ideals that appear in the algorithm before taking radicals. We call them $I_j^{\sharp}$. Kohn's ideals of multipliers $I_j$ are obtained by taking real radicals of the $I_j^{\sharp}$ in the ring of germs of smooth functions at a point. 

 Proposition \ref{mainpro} exhibits a collection of pseudoconvex domains in $\C^2$ where the order of the radical needed when passing from $I_2^{\sharp}$ to $I_2$ cannot be estimated in terms of the type. Theorem \ref{mainteo} provides a modified but effective algorithm.

\begin{proposition}\label{mainpro}
Let $\Omega$ be a pseudoconvex domain in $\mathbb{C}^2$ defined locally near the origin by\[2\Real(z)+|w^{\tau}+z^kw^l|^2<0.\] Here $\tau, k, l$ are integers such that $k>\tau> l>0, \tau> 2$. The type of the boundary at $0$ is equal to $2\tau$. To obtain the ideal $I_2$ in the second step of the Kohn algorithm one needs to take a real radical of order at least $k$. In particular, the Kohn algorithm is not effective on $\Omega$.  
\end{proposition}

The domains described in Proposition \ref{mainpro} are similar to known examples of domains in higher dimensions where the Kohn algorithm is not effective \cite{CD10, S10, KZ18}. In those papers, the authors consider domains in $\C^n$, for $n\geq 3$, that are locally defined at $0$ by 
\begin{equation*}\label{special}
2\Real(z_n)+\sum_{j=1}^m|f_j(z_1,\dots,z_{n-1})|^2<0,
\end{equation*}
where the $f_j$ are holomorphic functions. On such domains the Kohn algorithm reduces to a procedure in the ring of germs of holomorphic functions at the origin \cite[Section 6.4.4]{D93}. Note that the functions $f_j$ do not depend on $z_n$, and the domains are therefore ``rigid", following the terminology of \cite{BRT85}. The domains considered in Proposition \ref{mainpro} and Theorem \ref{mainteo}, however, are not rigid in general. Hence, when applying the Kohn algorithm, we work in the ring of germs of real analytic functions, and take real rather than holomorphic radicals.

Proposition \ref{mainpro} establishes the ineffectiveness of the Kohn algorithm in $\C^2$. We next show how the algorithm can be modified to make it effective on a class of finite type domains with real analytic boundary satisfying a condition slightly stronger than pseudoconvexity. The condition is formulated in terms of a ``holomorphic decomposition" \eqref{define} of a defining function \cite[Section 3.3.1]{D93}. (At the end of the paper we show how to generalize to higher dimensions).

Recall that for a domain $\Omega$ in $\C^2$ with real analytic boundary $b\Omega$ and a point $p\in b\Omega$, there exists a choice of local coordinates such that $p$ is the origin and $\Omega$ is defined near $0$ by
\begin{equation}\label{define}
2\Real(z)+\norm{f}^2-\norm{g}^2<0.
\end{equation}
Here $f=(f_j)_{j=1}^{\infty}$ and $g=(g_j)_{j=1}^{\infty}$ are countable sequences of holomorphic functions defined in a neighborhood $U$ of the origin, with $\norm{f}^2=\sum_{j=1}^{\infty}|f_{j}|^2$ and $\norm{g}^2=\sum_{j=1}^{\infty}|g_{j}|^2$ real analytic functions in $U$.
We denote by $f_w$ and $g_w$ the sequences of partial derivatives $(\partial_wf_j)_{j=1}^{\infty}$ and $(\partial_wg_j)_{j=1}^{\infty}$ respectively.
Assume now that there exists $\delta\in[0,1)$ such that 
\begin{equation}\label{hypo}
\norm{g_w}^2\leq\delta\norm{f_w}^2 \text{ near }0.
\end{equation}
If \eqref{hypo} holds, then $\Omega$ is pseudoconvex near $0$ (Lemma \ref{mainlemma}). Moreover, a modified version of the Kohn algorithm is effective on $\Omega$. The new procedure involves a single real radical of order two at the first step.

\begin{theorem}\label{mainteo}
Let $\Omega$ be a domain in $\C^2$ with real analytic boundary $b\Omega$, and let $p\in b\Omega$ be a point of finite type $2\tau$. Choose coordinates so that $p$ is the origin and $\Omega$ is locally defined near $0$ by \eqref{define}. Assume that \eqref{hypo} holds in a neighborhood of the origin. Then there exists an effective algorithm to establish subellipticity at $0$ for $\Omega$. In particular, the procedure yields a sequence $\zeta_1,\dots, \zeta_{\tau}$ of subelliptic multipliers in the ring of germs $\mathcal{A}_0$ of real analytic functions at the origin, with $\zeta_{\tau}$ a unit in $\mathcal{A}_0$. The element $\zeta_1$ is obtained by taking a real radical of order 2, while for $j>1$ we have $\zeta_j=\partial_w\zeta_{j-1}$. 
\end{theorem}

Our effective algorithm yields, for a domain of type $2\tau$, the lower bound  $(2^{\tau+1})^{-1}$ for the order of subellipticity at the origin. While very far from the known optimal bound of $(2\tau)^{-1}$, it seems to be the best effective bound that one can expect to obtain using subelliptic multipliers. 


The paper is organized as follows. In Section \ref{s2} we provide the necessary background and describe the Kohn algorithm. Section \ref{s3} deals with the phenomenon of ineffectiveness in $\C^2$ and contains the proof of Proposition \ref{mainpro}. In Section \ref{s4} we prove Theorem \ref{mainteo} and discuss the generalization to higher dimensions.

\section{The Kohn Algorithm}\label{s2}
We recall the algorithm for subelliptic multipliers introduced by Kohn in \cite{K79}. All the lemmas in this section are due to Kohn and are essentially contained in \cite[Proposition 4.7]{K79}. A more detailed exposition of the same material can be found in \cite[Chapter 6]{D93}. 

Let $\Omega$ be a bounded pseudoconvex domain in $\C^n$ with smooth boundary $b\Omega$. Recall that a {\em subelliptic estimate} holds on $(0,1)$ forms at a boundary point $p\in b\Omega$ if there exist a neighborhood $U$ of $p$ and positive constants $C,\epsilon$ such that the inequality
\begin{equation}\label{sub}
\norm{\phi}^2_{\epsilon}\leq C \Big{(}\norm{\bar\partial\phi}^2+\norm{\bar\partial^*\phi}^2+\norm{\phi}^2\Big{)}
\end{equation}
holds for all $(0,1)$ forms $\phi$ compactly supported in $U$ and in the domain of $\bar\partial^*$. Here $\bar\partial^*$ is the adjoint of $\bar\partial$ with respect to the standard $L^2$ inner product on $\Omega$. We denote by $\norm{\,\cdot\,}_{\epsilon}$ and $\norm{\,\cdot\,}$ the tangential Sobolev norm of order $\epsilon$ and the standard $L^2$ norm on $\Omega$ respectively.
We call the supremum of all the $\epsilon$ for which the estimate \eqref{sub} holds the {\em order of subellipticity at p}. 

\begin{remark}
By \cite[Proposition 3.10]{K79} one can use on the left side of \eqref{sub} either the full Sobolev $\epsilon$-norm or the tangential Sobolev $\epsilon$-norm (for the same $\epsilon$). The resulting inequalities are equivalent.
\end{remark}
Kohn defined the germ of a smooth function $f$ at $p$ to be a subelliptic multiplier if an estimate of the form \eqref{sub} holds when the $\phi$ on the left side of the inequality is replaced by $f\phi$. More precisely, $f$ is a {\em subelliptic multiplier} if there exist a neighborhood $U$ of $p$ and positive constants $C,\epsilon$ such that the inequality
\begin{equation*} 
\norm{f\phi}^2_{\epsilon}\leq C \Big{(}\norm{\bar\partial\phi}^2+\norm{\bar\partial^*\phi}^2+\norm{\phi}^2\Big{)}
\end{equation*}
holds for all $(0,1)$ forms $\phi$ compactly supported in $U$ and in the domain of $\bar\partial^*$. The supremum of all possible choices of $\epsilon$ is called the {\em order of the subelliptic multiplier $f$ at $p$.}

 Note that subellipticity at a point $p\in b\Omega$ is equivalent to $1$ being a subelliptic multiplier at $p$. Moreover, the order of subellipticity at $p$ is the same as the order of the subelliptic multiplier 1 at $p$.

\begin{remark}
Throughout this section, following \cite{K79}, we denote by $f$ the germ of a smooth function that is a subelliptic multiplier. This notation should not be confused with the $f$ arising from a holomorphic decomposition \eqref{define} of a defining function of the domain.
\end{remark}

\begin{lemma}\label{lemmar}
Let $f$ be a smooth function that vanishes on $b\Omega$. Then the germ of $f$ at a point $p\in b\Omega$ is a subelliptic multiplier of order $\epsilon=1$. 
\end{lemma}

Let $C^{\infty}_p$ denote the local ring of germs of smooth functions at $p$, and let $I$ be an ideal in $C^{\infty}_p$. Its real radical $\rad_{\R}I$ consists of all $g\in C^{\infty}_p$ such that $|g|^k\leq |f|$ near $p$ for some positive integer $k$ and some $f\in I$. We say that the real radical $\rad_{\R}I$ is {\em of order m} if $m$ is the least integer $M$ for which
 \[\rad_{\R}I=\Big{\{}g\in C^{\infty}_p\, \big{|} \,|g|^k\leq |f| \text{ near $p$ for some $f\in I$ and some positive integer $k\leq M$}\Big{\}}.\]

\begin{lemma}\label{lemmarad}
The collection $\mathcal{E}_p$ of elements of $C^{\infty}_p$ that are subelliptic multipliers is an ideal in $C^{\infty}_p$. Moreover, if $g\in C^{\infty}_p$ is such that $|g|^k\leq |f|$ near $p$ for some subelliptic multiplier $f\in C^{\infty}_p$ of order $\epsilon$, then $g$ is a subelliptic multiplier of order $\epsilon/k$. In particular, $\mathcal{E}_p$ is a real radical ideal, that is, $\mathcal{E}_p=\rad_{\R}\mathcal{E}_p$.
\end{lemma}

 In order to describe Kohn's procedure to generate subelliptic multipliers, we exploit the notion of vector multiplier. This terminology is due to Siu \cite{S10}, although the concept is already present in \cite{K79}. Vector multipliers are sometimes called ``allowable one forms" \cite{D93}.
 
  We say that a $(1,0)$ form $v=\sum_{j=1}^n v_jdz_j$ with smooth coefficients $v_j$ defined in a neighborhood of $p$ is a {\em vector multiplier} if there exist a neighborhood $U$ of $p$ and positive constants $C,\epsilon$ such that the inequality
\begin{equation}\label{eqvectmult}
\norm{\sum_{j=1}^n v_j\phi_j}^2_{\epsilon}\leq C \Big{(}\norm{\bar\partial\phi}^2+\norm{\bar\partial^*\phi}^2+\norm{\phi}^2\Big{)}
\end{equation}
holds for all $(0,1)$ forms $\phi=\sum_{j=1}^n\phi_j d\bar{z}_j$ compactly supported in $U$ and in the domain of $\bar\partial^*$. The supremum of all possible choices of $\epsilon$ is called the {\em order of subellipticity of $v$ at $p$.}

\begin{lemma}\label{lemmavectormult}
If $f\in C^{\infty}_p$ is a subelliptic multiplier of order $2\epsilon$, then $\partial f$ is a vector multiplier of order $\epsilon$.
\end{lemma}
It is convenient to express vector fields and forms in local coordinates modeled on the domain $\Omega$. Let $L_1,\dots, L_n$ be a collection of smooth $(1,0)$ vector fields defined in a neighborhood $U$ of $p$, and denote by $\omega_1,\dots,\omega_n$ the dual basis of $(1,0)$ forms. That is, $\langle\omega_i,L_j\rangle_z=\delta_{ij}$ for every $z\in U$, where $\langle\, ,\,\rangle_z$ indicates the pairing between a form and a vector field at $z$. It is standard in $CR$ geometry to choose these vector fields such that the following properties are satisfied:
\begin{itemize}
\item At every point $z\in U$ the vector fields $L_1,\dots, L_{n}$ form an orthonormal basis of $T^{1,0}_z\C^n$, with $L_1,\dots, L_{n-1}$ being an orthonormal basis for $T^{1,0}b\Omega$.
\item $\omega_n=\partial r$.
\item For $z\in U$ and $i,j=1,\dots,n$, we have $\langle \omega_i,\omega_j\rangle_z=\delta_{ij}$. Here $\langle\, ,\, \rangle_z$ denotes the inner product induced on the vector space of $(1,0)$ forms at $z$ by the Hermitian metric.
\end{itemize}
In the local frame just described, the condition that a $(0,1)$ form is in the domain of $\bar\partial^*$ becomes easy to write. In fact 
\begin{equation*}
\sum_{j=1}^n\varphi_j \bar{\omega}_j\in \Dom(\bar\partial^*)\Longleftrightarrow \varphi_n=0\text{ on }b\Omega.
\end{equation*} 
Note that by Lemma \ref{lemmar} the function $\varphi_n$ is a subelliptic multiplier. Let $v$ be a vector multiplier. Since the component $\langle v,\omega_n\rangle$ in the estimate \eqref{eqvectmult} is multiplied by $\varphi_n$, it is natural to only keep track of the components $\langle v,\omega_j\rangle$ for $j=1,\dots,n-1$. We thus make the following definition.

We call a collection $(g_1,\dots,g_{n-1})$ of germs of smooth functions at $p$ an {\em allowable row of order} $\epsilon$ if the $(1,0)$ form $\sum_{j=1}^{n-1}g_j\omega_j$ is a vector multiplier of order $\epsilon$. Since for a germ $f\in C_{p}^{\infty}$ we can write $\partial f=\sum_{j=1}^{n}(L_j f)\,\omega_j$, Lemma \ref{lemmavectormult} can be restated as follows.

\begin{lemma}\label{l1}
If $f\in C^{\infty}_p$ is a subelliptic multiplier of order $2\epsilon$, then $(L_1f,\dots, L_{n-1}f)$ is an allowable row of order $\epsilon$.
\end{lemma}
We say that a matrix is {\em allowable of order} $\epsilon$ if its rows are allowable and $\epsilon$ is the minimum among the orders of the rows.

We denote by $\lambda=(\lambda_{ij})$ the matrix of the Levi form in the basis $L_1,\dots, L_{n-1}$. That is, for $z$ near $p$, we have
\begin{equation*}
\lambda_{ij}(z)=\langle\partial\bar\partial r, L_i\wedge \bar{L}_{j}\rangle_z,\quad i,j=1,\dots,n-1.
\end{equation*}
\begin{lemma}\label{lemmalambda}
The matrix $\lambda$ is allowable of order $1/2$. The determinant of an allowable matrix of order $\epsilon$ is a subelliptic multiplier of order $\epsilon$.
\end{lemma}
 
Combining the results stated above, Kohn formulated a procedure to generate multipliers. His algorithm yields an increasing sequence 
\begin{equation*}
I_1\subseteq I_2\subseteq I_3\subseteq\dots
\end{equation*}
of real radical ideals $I_k$ in $C^{\infty}_p$ consisting of subelliptic multipliers. The starting point is the allowable matrix $\lambda$. We call $M_0$ the collection of rows of $\lambda$. The first ideal of multipliers $I_1$ is defined as
\begin{equation*}
 I_1=\rad_{\R}I_1^{\sharp},\quad\text{where}\quad I_1^{\sharp}=\mathcal{I}(r,\det\lambda).
\end{equation*} 
Here $\mathcal{I}(\,\,)$ denotes the ideal generated in $C^{\infty}_p$ by the elements appearing inside the parentheses.
Exploiting Lemma \ref{l1}, we add new allowable rows to the set $M_0$, thus obtaining  \[M_1=M_0\cup\,\big{\{} (L_1f,\dots,L_{n-1}f)\,\vert\, f\in I_1\big{\}}.\] Taking determinants, we can now produce more multipliers. We define inductively, for $j\geq 2$,
\begin{equation*}
I_j^{\sharp}=\mathcal{I}(I_{j-1},\det M_{j-1}),\quad I_j=\rad_{\R}I_j^{\sharp}.
\end{equation*}
\begin{equation*}
M_j=M_{j-1}\cup\,\big{\{}(L_1f,\dots,L_{n-1}f)\,\vert\, f\in I_{j-1}\big{\}}.
\end{equation*}
Here $\det M_j$ denotes all determinants of $(n-1)\times(n-1)$ matrices whose rows belong to $M_j$. 

 When $b\Omega$ is real analytic, subellipticity at a boundary point $p$ is completely characterized by whether the Kohn algorithm yields 1 as a subelliptic multiplier in finitely many steps. 
 \begin{theorem}\cite[Theorem 6.27]{K79}
Let $\Omega$ be a bounded pseudoconvex domain in $\C^n$ with real analytic boundary $b\Omega$, and let $p\in b\Omega$. Then the following statements are equivalent.
\begin{enumerate}
\item A subelliptic estimate holds at $p$.
\item $1\in I_k$ for some $k$.
\item There is no germ of a complex analytic variety at $p$ lying in $b\Omega$.
\end{enumerate}
\end{theorem} 
Under the assumption that the boundary is real analytic, condition {\em (3)} is equivalent to $b\Omega$ being of finite type at $p$ in the sense of D'Angelo \cite[Theorem 4.4]{D93}. 
While the equivalence of finite type and subellipticity is known to hold more generally in the $C^{\infty}$ case \cite{C83, C87}, it remains an open problem whether condition {\em (2)} is necessary for subellipticity when the boundary is smooth but not real analytic. 

We clarify that by the {\em type} of a real hypersurface at a point we mean the maximum order of contact with one-dimensional complex varieties as defined by D'Angelo \cite{D82}. We refer to \cite{D93} for precise definitions and more information on domains of finite type. We remark that there are other ways to measure the type, but in $\C^2$ they all coincide \cite[Section 4.3.1]{D93}. 

The equivalence between finite type and subellipticity was proved for domains with smooth boundary in $\C^2$  before Kohn developed his theory of subelliptic multipliers. The following statement combines the work of several authors \cite{K72,G74,RS76} (see also \cite[Section 3]{CD10} and \cite{Kr79}).
\begin{theorem}\label{T2}
Let $\Omega$ be a smoothly bounded pseudoconvex domain in $\C^2$, and let $p\in b\Omega$ be a boundary point. The following are equivalent:
\begin{itemize}
\item There is a subelliptic estimate at $p$ with $\epsilon=1/m$ but for no larger value of $\epsilon$.
\item The type of the boundary at $p$ is equal to $m$.
\end{itemize}
\end{theorem}

Theorem \ref{T2} establishes in $\C^2$ a precise relation between the order $\epsilon$ of subellipticity at $p$ and the measure of the type of $b\Omega$ at $p$. In the next section we show that no lower bound for the order of subellipticity in terms of the type can be obtained through the Kohn algorithm. This ``lack of effectiveness" of Kohn's procedure is well known in higher dimensions \cite{CD10}.  

 \section{Ineffectiveness of the Kohn Algorithm in $\C^2$}\label{s3}
We work in $\C^2$, with variables $z$ and $w$. We will often use subscript notation for derivatives, that is, for a smooth function $f$, we write $f_z$ to denote its partial derivative with respect to $z$.

 Let $\Omega$ be a domain with smooth boundary $b\Omega$ defined locally at $0$ by $r<0$, where $r$ is a smooth function with $r_z(0)\neq 0$. We consider the standard 
local basis for the bundle $T^{1,0}b\Omega$ given by the tangential vector field $L=\partial_w-(r_w/r_z)\partial_z$. Note that in dimension two the definitions of allowable row and subelliptic multiplier coincide. In particular, the matrix of the Levi form in the basis $L$ consists of one single function, which we call $\lambda$. 

Since we consider only domains with real analytic boundary, we carry out our computations in the ring of germs of real analytic functions $\mathcal{A}_0$. Accordingly, the notation $\mathcal{I}(\,\,)$ stands for the ideal generated in $\mathcal{A}_0$ by the elements appearing inside the parentheses. 

\begin{lemma}\label{reallemma}
Consider a pseudoconvex domain $\Omega$ in $\C^2$ locally defined near the origin by
\begin{equation}\label{defining}
2\Real(z)+|f(z,w)|^2<0,
\end{equation}
where $f$ is a holomorphic function defined in a neighborhood of $0$. The first step of the Kohn algorithm gives the ideals of subelliptic multipliers \[I_1^{\sharp}=\mathcal{I}\big{(} 2\Real(z)+|f|^2,|f_w|^2\big{)},\quad I_1=\rad_{\R}I_1^{\sharp}.\] 
\end{lemma}
\begin{proof}
Let $r$ be a smooth local defining function at $0$ for a domain in $\C^2$, with $r_z(0)\neq 0$. By Proposition 1 in Chapter 3 of \cite{D93} we have
\begin{equation}\label{leviprima}
\lambda=r_{w\bar{w}}r_zr_{\bar{z}}+r_{z\bar{z}}r_wr_{\bar{w}}-2\Real(r_{z\bar{w}}r_wr_{\bar{z}}).
\end{equation}
For $r=2\Real(z)+|f(z,w)|^2$, where $f$ is a holomorphic function, \eqref{leviprima} gives
\begin{equation}\label{2.3}
\lambda=|f_w|^2\Big{(}|1+f_z\bar{f}|^2+|ff_z|^2-2\Real\big{[}(1+f_z\bar{f})ff_{\bar{z}}\big{]}\Big{)}=|f_w|^2.
\end{equation}
Hence the first step of the Kohn algorithm yields \[I_1^{\sharp}=\mathcal{I}\big{(} r, \lambda\big{)}=\mathcal{I}\big{(} 2\Real(z)+|f|^2,|f_w|^2\big{)}.\] 

\end{proof}

\begin{proof}[Proof of Proposition \ref{mainpro}]
Recall that, of the curves with maximal order of contact, at least one must lie in the holomorphic tangent space \cite[page 128]{D93}. Hence the complex line $\gamma(t)=(0,t)$ achieves the maximum order of contact at $0$ with the boundary of $\Omega$. The type at $0$ is therefore equal equal to $2\tau$, and in particular is independent of $k$. We now show that the lower bound for the order of subellipticity given by the Kohn algorithm depends on $k$.

Let $r=2\Real(z)+|f|^2$, where $f=w^{\tau}+z^kw^l$. By Lemma \ref{reallemma} we have
\begin{equation*}
I_1^{\sharp}=\mathcal{I}\big{(} 2\Real(z)+|f|^2,|f_w|^2\big{)}.
\end{equation*}
Note that $f_w=w^{l-1}h$, with $h=\tau w^{\tau-l}+lz^k$. Hence 
\begin{equation}\label{3.4}
\mathcal{I}(2\Real(z)+|f|^2,wh,\bar{w}\bar{h})\subseteq I_1=\rad_{\R}I_1^{\sharp}.
\end{equation}
The elements of $I_1$ are all the germs in $\mathcal{A}_0$ that vanish on the common zeros of $2\Real(z)+|f|^2$ and $wh$. Taking into account that $wh$ does not divide $f$, one can prove that such elements have to be in $\mathcal{I}(2\Real(z)+|f|^2,wh,\bar{w}\bar{h})$. Hence equality holds in \eqref{3.4}.

 We know that $r$ is a multiplier of order $1$ (Lemma \ref{lemmar}) and $\lambda=|f_w|^2$ is a multiplier of order $1/2$ (Lemma \ref{lemmalambda}). A general element of $I_1^{\sharp}$ is therefore a multiplier of order at least $1/2$. Since $I_1$ is obtained by taking a radical of $I_1^{\sharp}$ of order $2(l-1)$, Lemma \ref{lemmarad} implies that a general element of $I_1$ is a multiplier of order at least $1/(4l-4)$.
 
 At the next step of the algorithm, we form the ideal $I_2^{\sharp}$ by adding to the list of generators of $I_1$ all the expressions of the form $Lh$, where $h\in I_1$. It is readily proved that it is enough to consider $h$ belonging to a set of generators of $I_1$. Note that $r_w=f_w\bar{f}\in I_1\subset I_2^{\sharp}$. Since $Lh=h_w+(r_z)^{-1}(r_wh_z)$, we have that $Lh\in I_2^{\sharp}$ implies $h_w\in I_2^{\sharp}$. Hence
\begin{equation*}
\begin{split}
I_2^{\sharp}&=\mathcal{I}\big{(}2\Real(z)+|w^{\tau}+z^kw^l|^2,\tau w^{\tau-l+1}+lz^kw,\tau \bar{w}^{\tau-l+1}+l\bar{z}^k\bar{w}, (\tau-l+1)\tau w^{\tau-l}+lz^k\big{)}.
\end{split}
\end{equation*}
It follows from the identity 
 \begin{equation*}
 (-\tau^2+\tau l)w^{\tau-l+1}=\tau w^{\tau-l+1}+lz^kw-w\big{[} (\tau-l+1)\tau w^{\tau-l}+lz^k\big{]}
 \end{equation*}
 that we can rewrite $I_2^{\sharp}$ as
 \[ I_2^{\sharp}=\mathcal{I}\big{(}2\Real(z)+|w^{\tau}+z^kw^l|^2, w^{\tau-l+1},\tau \bar{w}^{\tau-l+1}+l\bar{z}^k\bar{w}, (\tau-l+1)\tau w^{\tau-l}+lz^k\big{)}.\]
 
 By Lemma \ref{l1}, the elements of $I_2^{\sharp}$ are subelliptic multipliers of order at least $1/(8l-8)$. 
Since $ w^{\tau-l+1}\in I_2^{\sharp},$ we have $w\in I_2=\rad_{\R}I^{\sharp}_2$. Hence
\begin{equation*}
I_2=\rad_{\R}\mathcal{I}\big{(}2\Real(z), w,z^k\big{)}=\mathcal{I}\big{(}z,\bar{z}, w,\bar{w}\big{)}.
\end{equation*}
 Note that there exist no smooth functions $\alpha,\beta,\gamma$ such that 
\begin{equation}\label{order}
|z|^{k-1}\leq \alpha|\Real(z)|+\beta |z|^k+\gamma |w|
\end{equation}
holds in a neighborhood of the origin. Indeed, for every choice of $\alpha,\beta,\gamma$, the inequality \eqref{order} fails at points of $\C^2$ of the form $(ia,0)$ for $a$ a real number sufficiently close to zero. Hence a radical of order at least $k$ is needed to obtain $I_2$ from $I_2^{\sharp}$. By Lemma \ref{lemmarad}, the lower bound for the order of subellipticity of an element of $I_2$ drops to at least $1/(8lk-8k)$. 
\end{proof}

\begin{remark}
Following the computations above, one can prove that the  conclusions of Proposition \ref{mainpro} hold in a slightly more general setting. More precisely, for domains locally defined at the origin by $2\Real(z)+|w^{\tau}+z^kw^l+g(z)|^2<0$, where $g$ is a holomorphic function and $k>\tau> l>0, \tau> 2$, the second step of the Kohn algorithm requires a real radical of order at least $k$.
\end{remark}

\begin{remark}
It is easy to modify the Kohn algorithm to make it effective on the domains considered in Proposition \ref{mainpro}. One simply does not take any radical at the second step of the algorithm, and instead keeps producing new multipliers by applying the vector field $L$. Note that a real radical of order 2 in the first step is still needed. The effective procedure described in the next section generalizes this idea. 
\end{remark}


\section{An effective algorithm for a class of pseudoconvex domains in $\C^2$}\label{s4}
For a real analytic real hypersurface in $\C^n$, there is a convenient way to write a local defining equation in terms of absolute values of holomorphic functions. We briefly recall the results needed for our study of domains with real analytic boundary in $\C^2$. We refer to \cite[Section 3.3.1]{D93} for the general statements and proofs concerning the {\em holomorphic decomposition} of a defining function.

We start by introducing some notation. For an open set $U$ in $\C^2$, we denote by $l^2(U)$ the Hilbert space of square summable countable sequences of holomorphic functions in $U$. That is, an element $f\in l^2(U)$ is a sequence $(f_j)_{j=1}^{\infty}$ of holomorphic functions in $U$ such that 
\begin{equation}\label{norm}
\norm{f}^2=\sum_{j=1}^{\infty}|f_j|^2\,\,\text{ converges in } U.
\end{equation} 

\begin{remark}
Note that the norm just introduced does not involve integration, and should not be confused with the $L^2$ norm used in Section \ref{s2}.
\end{remark}

 The norm in \eqref{norm} is induced by a Hermitian inner product, denoted by $\langle\,,\,\rangle.$ For elements $f=(f_j)_{j=1}^{\infty}$ and $g=(g_j)_{j=1}^{\infty}$ in $l^2(U)$ we have
\begin{equation*}
\langle f,g\rangle=\sum_{j=1}^{\infty} f_j\bar{g}_j.
\end{equation*}

Given a bounded domain $\Omega$ in $\C^2$ with real analytic boundary $b\Omega$ and a point $p\in b\Omega$, we can always choose coordinates so that $p$ is the origin, and in a neighborhood $U$ of $0$ the domain $\Omega$ is defined by
\begin{equation}\label{defequ}
2\Real(z)+\norm{f}^2-\norm{g}^2<0
\end{equation}
for some $f,g\in l^2(U)$.

We will consider domains with real analytic boundary where the elements $f$ and $g$ resulting from a holomorphic decomposition \eqref{defequ} satisfy an additional assumption. Recall that for $h=(h_j)_{j=1}^{\infty}\in l^2(U)$ we denote by $h_w$ the element of $l^2(U)$ given by $(\partial_w h_j)_{j=1}^{\infty}.$

\begin{lemma}\label{mainlemma}
Let $\Omega$ be a domain in $\C^2$ defined in a neighborhood $U$ of $0$ by $r<0$, where
\begin{equation*}
r=2\Real(z)+\norm{f}^2-\norm{g}^2
\end{equation*} for some element $f,g\in l^2(U)$. Assume that there exists $\delta\in [0,1)$ such that, in $U$,
\begin{equation}\label{hypodue}
\norm{g_w}^2\leq\delta\norm{f_w}^2.
\end{equation} Then there exist a positive constant $C$ and a neighborhood $V$ of the origin in $\C^2$ such that
\begin{equation}\label{conclu}
\lambda\geq C\norm{f_w}^2 \text{ in } V.
\end{equation}
In particular, $\Omega$ is pseudoconvex near $0$.
\end{lemma}
\begin{remark}\label{Examplenn}
Condition \eqref{hypodue} is strictly stronger than pseudoconvexity. Consider for instance a domain in $\C^2$ defined near the origin by
\[2\Real(z)+|w+w^k|^2+|w^2|^2-|w|^2<0,\]
where $k>4$ is an integer. For this domain we have \[f=(w+w^k,w^2),\quad g=(w),\]
and therefore \begin{equation*}
\norm{f_w}^2=|1+kw^{k-1}|^2+|2w|^2,\quad \norm{g_w}^2=1.
\end{equation*}
Assume that there exists a positive constant $\delta$ such that near the origin we have
\begin{equation}\label{gotozero}
1=\norm{g_w}^2\leq \delta\norm{f_w}^2=\delta \Big(|1+kw^{k-1}|^2+|2w|^2\Big).
\end{equation}
Letting $|w|\to 0$ in \eqref{gotozero}, we see that $\delta\geq 1$. Hence the domain does not satisfy condition \eqref{hypodue}. 
A computation of the Levi form $\lambda$ yields
\[\lambda=\norm{f_w}^2-\norm{g_w}^2=2\Real(kw^{k-1})+k^2|w^{k-1}|^2+4|w|^2,\]
which is non-negative for $w$ close to $0$. The domain is therefore pseudoconvex near the origin.
\end{remark}

\begin{proof}[Proof of Lemma \ref{mainlemma}]
By Proposition 1 in Chapter 3 of \cite{D93} we have
\begin{equation}\label{leviseconda}
\lambda=r_{w\bar{w}}r_zr_{\bar{z}}+r_{z\bar{z}}r_wr_{\bar{w}}-2\Real(r_{z\bar{w}}r_wr_{\bar{z}}).
\end{equation}
We compute each of the terms in \eqref{leviseconda}. The first one is
\begin{equation*}
r_{w\bar{w}}r_zr_{\bar{z}}=\big(\norm{f_w}^2-\norm{g_w}^2\big)\big{|}1+\langle f_z,f \rangle-\langle g_z,g \rangle\big{|}^2.
\end{equation*}
Exploiting \eqref{hypodue}, we obtain the estimate
\begin{equation}\label{prima}
r_{w\bar{w}}r_zr_{\bar{z}}\geq \norm{f_w}^2(1-\delta)|1+H|^2,
\end{equation}
where $H=\langle f_z,f \rangle-\langle g_z,g \rangle.$ Note that $H(0)=0$. We now compute the second term of \eqref{leviseconda}.
\begin{equation*}
\begin{split}
r_{z\bar{z}}r_wr_{\bar{w}}&=\big(\norm{f_z}^2-\norm{g_z}^2\big)\big{|}\langle f_w,f \rangle-\langle g_w,g \rangle\big{|}^2\\&=\big(\norm{f_z}^2-\norm{g_z}^2\big)\big( |\langle f_w,f\rangle|^2+|\langle g_w,g \rangle|^2-\langle f_w,f\rangle\langle g,g_w\rangle-\langle g_w,g\rangle\langle f,f_w\rangle\big).
\end{split}
\end{equation*}
The Cauchy-Schwarz inequality gives
\begin{equation}\label{y}
|r_{z\bar{z}}r_wr_{\bar{w}}|\leq \big{|}\norm{f_z}^2-\norm{g_z}^2\big{|}\big( \norm{f_w}^2\norm{f}^2+\norm{g_w}^2\norm{g}^2+2\norm{f_w}\norm{g_w}\norm{f}\norm{g}\big).
\end{equation}
By \eqref{hypodue} we have
\begin{equation}\label{star}
2\norm{f_w}\norm{g_w}\leq \norm{f_w}^2+\norm{g_w}^2\leq (1+\delta)\norm{f_w}^2.
\end{equation}
Combining \eqref{star} and \eqref{hypodue} with \eqref{y} yields
\begin{equation*}
|r_{z\bar{z}}r_wr_{\bar{w}}|\leq \big{|}\norm{f_z}^2-\norm{g_z}^2\big{|}\norm{f_w}^2 \big(\norm{f}^2+\norm{g}^2\delta+\norm{f}\norm{g}(1+\delta)   \big).
\end{equation*}
We conclude that
\begin{equation}\label{seconda}
|r_{z\bar{z}}r_wr_{\bar{w}}|\leq \norm{f_w}^2 K,
\end{equation}
for some $K\in\mathcal{A}_0$ with $K(0)=0$.
For the last term in \eqref{leviseconda}, we have
\begin{equation*}
r_{z\bar{w}}r_wr_{\bar{z}}=\big(\langle f_z,f_w\rangle-\langle g_z,g_w\rangle\big)\big(\langle f_w,f\rangle-\langle g_w,g\rangle\big)\big(1+\langle f,f_z\rangle-\langle g,g_z\rangle\big).
\end{equation*}
Distributing the first product, we get 
\begin{equation*}
\begin{split}
|r_{z\bar{w}}r_wr_{\bar{z}}|= &\big{|}\langle f_z,f_w\rangle \langle f_w,f\rangle+\langle g_z,g_w\rangle \langle g_w,g\rangle\\&-\langle
f_z,f_w\rangle \langle g_w,g\rangle-\langle g_z,g_w\rangle \langle f_w,f\rangle\big{|} \big{|}1+\langle f,f_z\rangle-\langle g,g_z\rangle\big{|}.\\
\end{split}
\end{equation*}
The Cauchy-Schwarz inequality then yields
\begin{equation*}
\begin{split}
|r_{z\bar{w}}r_wr_{\bar{z}}|\leq &\Big(\norm{f_z}\norm{f_w}^2\norm{f}+\norm{g_z}\norm{g_w}^2\norm{g}\\&+\norm{f_w}\norm{g_w}\big(\norm{f_z}\norm{g}+\norm{g_z}\norm{f}\big)\Big)\big{|}1+\langle f,f_z\rangle-\langle g,g_z\rangle\big{|}.
\end{split}
\end{equation*}
Exploiting \eqref{hypodue} and \eqref{star}, we obtain
\begin{equation}\label{terza}
|r_{z\bar{w}}r_wr_{\bar{z}}|\leq \norm{f_w}^2 R,
\end{equation}
where $R\in\mathcal{A}_0$ with $R(0)=0$.
Choose a small $\epsilon>0$, and let $V$ be a neighborhood of the origin such that $|H|<\epsilon$ in $V$ and $|K|,|R|<\epsilon(1-\delta)$ in $V$.
Combining \eqref{prima}, \eqref{seconda}, and \eqref{terza} with \eqref{leviseconda}, we have that in $V$ the following inequality holds:
\begin{equation*}
\lambda\geq \norm{f_w}^2(1-\delta)(1-\epsilon)-3\norm{f_w}^2\epsilon(1-\delta)=\norm{f_w}^2(1-\delta)(1-4\epsilon). 
\end{equation*}
Letting $C=(1-\delta)(1-4\epsilon)$, we obtain \eqref{conclu}, as desired.
\end{proof}
\begin{proof}[Proof of Theorem \ref{mainteo}]

By Lemma \ref{mainlemma} there exists a neighborhood $V$ of the origin in $\C^2$ and a constant $C>0$ such that, for every $j$, we have
\begin{equation*}
|\partial_wf_j|^2\leq C^{-1} \lambda \quad \text{in }V.
\end{equation*} 
By Lemmas \ref{lemmarad} and \ref{lemmalambda}, each element $\partial_wf_j$ is a subelliptic multiplier of order $1/4$. By hypothesis, $\norm{g_w}^2<\norm{f_w}^2$ near $0$. Hence, after possibly shrinking $V$, we have
\begin{equation*}
|\partial_wg_j|^2\leq C^{-1} \lambda \quad \text{in }V.
\end{equation*}
The elements $\partial_wg_j$ are therefore subelliptic multipliers of order $1/4$.
Recall that we can produce new multipliers by considering $Lh$, where $h$ is a multiplier and $L$ is the tangential vector field $L=\partial_w-(r_w/r_z)\partial_z$. Note that $r_w=\langle f_w,f\rangle-\langle g_w,g\rangle$ is itself a multiplier, being in the ideal generated in $\mathcal{A}_0$ by the elements $\partial_wf_j$ and $\partial_wg_j$. Hence $Lh$ being a multiplier implies that $\partial_w h$ is a multiplier. 

For a domain defined by \eqref{define}, the maximum order of contact is achieved by the complex line $\gamma(t)=(0,t)$. Let $r$ be the defining function appearing in \eqref{defequ}. Recall that the type of the boundary at $0$ is equal to $2\tau$.  By \cite[Proposition 2]{D93} the coefficient of $|t|^{2\tau}$ in $(r\circ\gamma)(t)$ is positive. Hence at least one of the functions $f_j$ vanishes to order $\tau$ at $0$ along $\gamma$. 

We are ready to describe our effective procedure to generate a sequence of subelliptic multipliers. Let $\zeta_1=\partial_wf_k$, where $f_k$ is a component of $f$ vanishing to order $\tau$ along $\gamma$. Now let $\zeta_j=\partial_w \zeta_{j-1}$ for $j\geq 2$. The arguments at the beginning of this proof show that every $\zeta_j$ is a subelliptic multiplier. By the choice of $f_k,$ the element $\zeta_{\tau}$ is a unit in $\mathcal{A}_0$, as well as being a multiplier of order at least $(2^{\tau+1})^{-1}$. 
\end{proof}
\section{Extension to higher dimensions} We indicate how to extend the results of this paper to higher dimensions. We work in $\C^{n+1}$ with variables $z,w_1,\dots,w_n$. For a domain with real analytic boundary in $\C^{n+1}$ locally defined by $r=0$, we consider a holomorphic decomposition 
\begin{equation*}
r=2\Real(z)+\norm{f(z,w)}^2-\norm{g(z,w)}^2,
\end{equation*} where $f$ and $g$ are square summable sequences of holomorphic functions defined in a neighborhood of $0$ in $\C^{n+1}$. Condition \eqref{hypo} is then replaced by 
\begin{equation}\label{newcond}
\big(\langle g_{w_i},g_{w_j}\rangle_{ij}\big)\leq\delta\big(\langle f_{w_i},f_{w_j}\rangle_{ij}\big)\text{ near $0$}\text{ for some }\delta\in[0,1).
\end{equation}
Here $\big(\langle g_{w_i},g_{w_j}\rangle_{ij}\big)$ and $\big(\langle f_{w_i},f_{w_j}\rangle_{ij}\big)$ stand for the Hessians in the $w$ variables of the real analytic functions $\norm{g}^2$ and $\norm{f}^2$ respectively, and the inequality is intended in the sense of matrices. As usual, for matrices $A$ and $B$ we write $A\geq B$ if and only if $A-B$ is positive semidefinite. In analogy with Lemma \ref{mainlemma}, condition \eqref{newcond} implies that there exist a positive constant $C$ and a neighborhood $V$ of the origin such that 
\begin{equation}\label{neq}
\big(\lambda_{ij}\big)\geq C\,\big(\langle f_{w_i}, f_{w_j}\rangle_{ij}\big) \text{ in } V.
\end{equation}
Hence the Hessian in the $w$ variables of $\norm{f}^2$ serves as a positive semidefinite lower bound for the Levi form. It follows that the domain is pseudoconvex near 0.

When $g\equiv 0$ and $f$ is independent of $z$ (the ``rigid sum of squares" case), the Hessian $\big(\langle f_{w_i}, f_{w_j}\rangle_{ij}\big)$ equals the Levi form $\big(\lambda_{ij}\big)$, which is an allowable matrix by Lemma \ref{lemmalambda}. Note that $\big(\langle f_{w_i}, f_{w_j}\rangle_{ij}\big)=(D_wf)^*D_wf$, where $D_wf$ is the Jacobian of $f$ in the $w$ variables and $(D_wf)^*$ is its Euclidean adjoint. It follows from the identity $\lambda=(D_wf)^*D_wf$ that the matrix $D_wf$ is also allowable (see \cite{K79} and \cite[Chapter 6]{D93}), and it becomes the starting point for the Kohn algorithm. 

Consider now the more general situation in which $g$ is not necessarily zero, but \eqref{newcond} is satisfied. One can prove using \eqref{neq} that the Jacobian matrix $D_wf$ is still allowable in this case (see \cite[Chapter 4]{F20} for details), and can therefore be taken as the starting point for a multiplier algorithm, as in the rigid sum of squares setting.

 In this paper we have exploited a known effective version of the Kohn algorithm for the rigid sum of squares case in $\C^2$ (consisting in only taking derivatives) to establish an effective procedure to prove subellipticity for domains of finite type with real analytic boundary satisfying \eqref{hypo}. In the author's PhD thesis \cite{F20} these ideas are generalized to higher dimensions in the way described above. It is shown that results on the effectiveness of the Kohn algorithm for the rigid sum of squares case can be translated to the more general setting of finite type domains with real analytic boundary satisfying \eqref{newcond}. 

\bibliographystyle{alpha}
 							\section{Acknowledgements}
 							This work constitutes part of the author's PhD thesis. He would like to thank his advisor John D'Angelo for many fruitful discussions, and for suggesting the example in Remark \ref{Examplenn}. The author also acknowledges many helpful comments from the referee. This research was conducted while the author was partially supported on a Golub Research Assistanship by the Mathematics Department of the University of Illinois at Urbana-Champaign and by NSF Grant DMS 13-61001 of John D'Angelo.
					
\end{document}